\documentclass[12pt]{article}
\usepackage[nottoc]{tocbibind}
\usepackage{mathtools,commath}
\usepackage{amssymb,amsmath,amsthm}
\usepackage{enumitem}
\usepackage{multirow}
\usepackage{bbm}
\usepackage{hyperref}
\usepackage{fullpage}
\newtheorem{thm}{Theorem}
\newtheorem{lemma}[thm]{Lemma}

\newtheorem*{remark}{Remark}
\newtheorem{cor}{Corollary}
\newtheorem*{definition}{Definition}

\newtheorem{theorem}{Theorem}[section]

\def\l{\lambda}

\def\s{\sigma}
\def\r{\rho}

\def\p{\varphi}
\def\bR{\mathbb R}

\def\DMO{\DeclareMathOperator}

\DMO{\GL}{GL}
\DMO{\SL}{SL}
\DMO{\gl}{\mathfrak{gl}}
\DMO{\ssl}{\mathfrak{sl}}
 
\def\imp{\ensuremath{\implies}}
\usepackage{subfiles}

\title{Blow-up radial solutions for elliptic systems with monotonic non-linearities}

\author{ Daniel Devine\footnote{School of Mathematics, Trinity College Dublin; {\tt dadevine@tcd.ie}} 
    $\;\;$        $\;$
Gurpreet Singh\footnote{School of Mathematical Sciences, Dublin City University, Ireland; {\tt gurpreet.bajwa2506@gmail.com}}}
\date{}

\begin{document}
\maketitle

\begin{abstract}
We are concerned with the existence and boundary behaviour of positive radial solutions for the system
\begin{equation*}
\left\{
\begin{aligned}
\Delta u&=g(|x|,v(x)) &&\quad\mbox{in}\ \Omega, \\
\Delta v&=f(|x|,|\nabla u(x)|) &&\quad\mbox{in}\ \Omega,
\end{aligned}
\right.
\end{equation*}
where $\Omega \subset \bR^N$ is either a ball centered at the origin or the whole space $\bR^N$, and $f,g\in C^{1}([0,\infty)\times [0,\infty))$, are non-negative, and increasing. Firstly, we study the existence of positive radial solutions in the case when the system is posed in a ball corresponding to their behaviour at the boundary. Next, we discuss the existence of positive radial solutions in case when  $g(|x|,v(x)) = |x|^{a} v^p$ and $f(|x|, |\nabla u (x)|) = |x|^{b} h(|\nabla u|)$. Finally, we take $h(t) = t^s$, $s> 1$, $\Omega = \bR^N$ and by the use of dynamical system techniques we are able to describe the behaviour at infinity of such positive radial solutions.
\end{abstract}

\smallskip


\section{Introduction}
In this paper we study the existence and behaviour of radially symmetric solutions to the system
\begin{equation}\label{system}
\left\{
\begin{aligned}
\Delta u&=g(|x|,v(x)) &&\quad\mbox{in}\ \Omega, \\
\Delta v&=f(|x|,|\nabla u(x)|) &&\quad\mbox{in}\ \Omega,
\end{aligned}
\right.
\end{equation}
where $\Omega \subset \bR^N$ is either a ball centered at the origin with radius $R> 0$ or the whole space $\bR^N$, and for all $t\in\bR, R>0$ we have
\begin{equation}\label{cond f,g finite}
\lim\limits_{r\to R^{-}}g(r,t)<\infty\qquad	\mbox{ and }\qquad \lim\limits_{r\to R^{-}}f(r,t)<\infty.
\end{equation}
We assume that $f,g\in C^{1}([0,\infty)\times [0,\infty))$, are non-negative, and are increasing in both variables. The study of systems such as \eqref{system} has generally focused on the case where $f,g$ are polynomial in both variables, motivating the above conditions. The case where $\Omega$ is a ball, $g(|x|,v(x))=v(x)$ and $f(|x|,|\nabla u|)=|\nabla u|^{2}$, that is
\begin{equation}\label{DLSsystem}
\left\{
\begin{aligned}
\Delta u&=v  &&\quad\mbox{ in } \Omega, \\
\Delta v&=|\nabla u|^{2} &&\quad\mbox{ in } \Omega,
\end{aligned}
\right.
\end{equation}
was first studied by Diaz, Lazzo and Schmidt in~\cite{DLS2005}. This choice of the functions $f$ and $g$ appears in the study of the dynamics of viscous, heat-conducting fluids.  In~\cite{DLS2005}, the authors
obtained that the system \eqref{DLSsystem} has one positive solution which blows up at the boundary and the authors also observed that in case of small dimensions $(N \leq 9)$, there exists one
sign-changing solution which also blows up at the boundary. In~\cite{DRS2007,DRS2008}, Diaz, Rakotoson, and Schmidt extended these results to time dependent systems. 

In~\cite{S2015}, Singh considered system \eqref{system} in the case where $g=v^{p}$ and $f=f(|\nabla u|)$, that is
\begin{equation}\label{Ssystem}
\left\{
\begin{aligned}
\Delta u&=v^{p}  &&\quad\mbox{ in } \Omega, \\
\Delta v&=f(|\nabla u|) &&\quad\mbox{ in } \Omega,
\end{aligned}
\right.
\end{equation}
Here, the author considered the system \eqref{Ssystem} both in the case of a ball centered at the origin with positive radius and in the whole space $\bR^{N}$ and classified all the positive radial solutions of \eqref{Ssystem} together with the behaviour of solutions at the boundary. The results obtained in \cite{S2015} have since been extended in \cite{DS2022} to the following weighted system:
\begin{equation}\label{system2}
\left\{
\begin{aligned}
\Delta u&=|x|^{a}v^{p}  &&\quad\mbox{ in } \Omega, \\
\Delta v&=|x|^{b}v^{q}f(|\nabla u|) &&\quad\mbox{ in } \Omega,
\end{aligned}
\right.
\end{equation}
where the authors again discussed the existence of positive radial solutions and boundary behaviour in case of a ball with positive radius, centered at the origin and in the whole space $\bR^{N}$.

Boundary blow-up problems have a long history and one could link them back to atleast a century ago. For example, in 1916, Bieberbach \cite{B1916} had studied the boundary blow-up solutions for the equation $\Delta u=e^u$ in a planar domain. Since then, there has been many new techniques which have been devised to tackle such solutions ( see \cite{GRbook2008,GRbook2012,Rbk}). Semilinear elliptic equations with nonlinear gradient terms have been studied extensively for such boundary blow-up solutions in last few decades ( for instance, see \cite{AGMQ2012,CPW2013,FQS2013, FV2017, GNR2002, MMMR2011}).

In \cite{GGS2019}, Ghergu, Giacomoni and Singh studied the following more general quasilinear elliptic system with nonlinear gradient terms 

\begin{equation}\label{gsys2}
	\left\{
	\begin{aligned}
		\Delta_{p} u&=v^{m} |\nabla u|^{\alpha} &&\quad\mbox{ in } \Omega, \\
		\Delta_{p} v&=v^{\beta}|\nabla u|^q &&\quad\mbox{ in } \Omega,
	\end{aligned}
	\right.
\end{equation}
in which the authors classified all the positive radial solutions in case $\Omega$ is a ball and also obtained the behaviour at infinity of such solutions.

We first consider the case when $\Omega = B_R$ is a ball of radius $R> 0$ and centered at the origin. We obtained that in this case the system \eqref{system} has positive radially symmetric solutions $(u, v)$ such that $u$ or $v$ ( or both ) blow up around $\partial{\Omega}$ if 
\begin{equation}\label{gsys3}
\int_{1}^{\infty}\frac{ds}{g\left(R,D^{-1}\left(C\left(\int_{0}^{s}\sqrt{f(\r,t)}dt\right)^{2}\right)\right)} <\infty,
\end{equation}
as well as a (sometimes equivalent) necessary condition for this behaviour. Next, we study the system 
\begin{equation*}
	\left\{
	\begin{aligned}
		\Delta u&=|x|^{a}v^{p}  &&\quad\mbox{ in } B_R, \\
		\Delta v&=|x|^{b}h(|\nabla u|) &&\quad\mbox{ in } B_R,
	\end{aligned}
	\right.
\end{equation*} 
where $h\in C^{1}[0,\infty)$ is an increasing function such that $h(t) > 0$ for all $t > 0$. We obtained that the above system has positive radially symmetric solutions $(u, v)$ such that $u$ or $v$ ( or both ) blow up around the boundary of $B_R$ if and only if

\begin{equation}\label{bound1}
	\int_{1}^\infty\frac{ds}{\Big(\displaystyle \int_0^s H(t)dt  \Big)^{p/(2p+1)}} < \infty, \quad\mbox{ where }\; H(t)=\int_0^t h(k)dk.
\end{equation}
Finally, in case when $h(t) = t^s$, $s \geq 1$, then we are able to find the exact 
rate at which the solution $(u, v)$ blows up at the infinity. In order to do that, we have used the dynamical system techniques for cooperative systems with negative divergence to obtain the exact blow up rate at infinity.

The conditions \eqref{gsys3} and \eqref{bound1} are similar to optimal conditions obtained by Keller \cite{K1954} and Osserman \cite{O1957} in 1950s while studying the existence of a solution to the boundary blow-up problem
\begin{equation}\label{gsys4}
	\left\{
	\begin{aligned}
		\Delta u&=f(u)&&\quad\mbox{ in }\Omega,\\
		u&=\infty &&\quad\mbox{ on }\partial\Omega,
	\end{aligned}
	\right.
\end{equation}
where $\Omega\subset \bR^N$ is a bounded domain and $f\in C^{1}[0,\infty)$ is a nonnegative increasing function. The authors obtained that \eqref{gsys4} has $C^2(\Omega)$ solutions if and only if
\begin{equation}\label{gsys5}
	\int_1^\infty\frac{ds}{\sqrt{F(s)}}<\infty\quad\mbox{ where }\; F(s)=\int_0^s f(t)dt.
\end{equation}
Equation \eqref{gsys5} has also been seen in various other circumstances as it is related to the maximum principle for nonlinear elliptic inequalities. For example, if $u\in C^2(\Omega)$ is nonnegative and satisfies $\Delta u\leq f(u)$ in $\Omega$, then, if $u$ vanishes at a point in $\Omega$, it must vanish everywhere in $\Omega$.



\section{Main Results}
Our first two results provide necessary and sufficient conditions under which positive radially symmetric solutions $(u,v)$ to \eqref{system} blow up around $\partial{\Omega}$, where $\Omega$ is a ball of radius $R$ centred at the origin. As we see in Remark 2.3 below, the conditions \eqref{int H finite} and \eqref{int I finite} are equivalent for functions $g$ which grow at most polynomially. We define the following two functions:
\begin{equation*}
\begin{aligned}
A_{\r}(s)&=\int_{0}^{s}\frac{g(\r,t)}{\sqrt{t}}dt\qquad\mbox{for all}\ 0<\r\leq R,\\
D(s)&= \int_{0}^{s}g^{2}(R,t)dt.
\end{aligned}
\end{equation*}
We assume for all $0<\r\leq R$
\begin{equation*}
\lim\limits_{s\to\infty}A_{\r}(s)=\lim\limits_{s\to\infty}D(s)=\infty.
\end{equation*}
Note that both functions defined above are strictly increasing in the variable $s$. Throughout this paper we let $C$ denote positive constants which may vary from line to line, or within the same line.

\smallskip

\begin{theorem}\label{newthm}
\begin{itemize}
\item[(i)] Suppose $(u,v)$ is a positive, radial solution to \eqref{system}. If 
\begin{equation*}
\lim\limits_{r\to R}v(r)=\infty,
\end{equation*}
then there exists $C>0$, such that
\begin{equation}\label{int H finite}
\int_{1}^{\infty}\frac{ds}{g\left(R, A_{r_{0}}^{-1}\left(C\int_{0}^{s}\sqrt{f(R,t)}dt\right)\right)}<\infty,
\end{equation}
for some $r_{0}<R$.
\item[(ii)] Suppose there exists $C>0$, $\r\leq R$ such that
\begin{equation}\label{int I finite}
\int_{1}^{\infty}\frac{ds}{g\left(R,D^{-1}\left(C\left(\int_{0}^{s}\sqrt{f(\r,t)}dt\right)^{2}\right)\right)} <\infty,
\end{equation}
then there exists a positive, radial solution $(u,v)$ to \eqref{system} satisfying
\begin{equation*}
\lim\limits_{r\to R}v(r)=\infty.
\end{equation*}
\item[(iii)] If system \eqref{system} has a positive radial solution satisfying
\begin{equation}\label{u,v unbound}
	\lim\limits_{|x|\to R}u(x)=\infty
\end{equation}
then there exist $C>0$, $\r\leq R$ such that the following holds:
\begin{equation*}
	\int_{1}^{\infty}\frac{s  ds}{g\left(R,D^{-1}\left(C\int_{0}^{s}\sqrt{f(\r,t)}dt\right)^{2}\right)}=\infty.
\end{equation*}
Also, if system \eqref{system} has a positive radial solution where 
\begin{equation}\label{u bounded}
	\mbox{u is bounded in}\ B_{R},
\end{equation}
then there exists $C>0$, $r_{0}\leq R$ such that the following holds:
\begin{equation*}
	\int_{1}^{\infty}\frac{s ds}{g(R,A^{-1}_{r_{0}}\left(C\int_{0}^{s}\sqrt{f(R,t)}dt\right)}<\infty.
\end{equation*}
\end{itemize}	
\end{theorem}

\smallskip

\begin{remark}
\normalfont{
Suppose our non-linearity $g(r,t)$ satisfies 
\begin{equation}\label{lim g}
\lim\limits_{t\to\infty}\frac{g(r,t)}{t^p}=\p(r)>0
\end{equation}
for some $p>0$, then conditions \eqref{int H finite} and \eqref{int I finite} are equivalent. In this case we find, for all large enough $s>0$,
\begin{equation*}
A_{R}(s)=A(s)=\int_{0}^{s}\frac{g(R,t)}{\sqrt{t}}dt \geq  \int_{\frac{s}{2}}^{s}\frac{g(R,t)}{\sqrt{t}}dt\geq   C\int_{\frac{s}{2}}^{s}\frac{t^{p}}{\sqrt{t}}dt\geq Cs^{\frac{2p+1}{2}}.
\end{equation*}
In other words, we have for large enough $s$ that $A^{-1}(Cs^{\frac{2p+1}{2}})\leq s$, which implies
\begin{equation*}
\begin{aligned}
g\left(R,A^{-1}\left(C\left(\left(\int_{0}^{s}\sqrt{f(R,t)}dt\right)^{\frac{2}{2p+1}}\right)^{\frac{2p+1}{2}}\right)\right)
&\leq g\left(R,\left(C\int_{0}^{s}\sqrt{f(R,t)}dt\right)^{\frac{2}{2p+1}}\right)\\
&\leq C\left(\int_{0}^{s}\sqrt{f(R,t)}dt\right)^{\frac{2p}{2p+1}}.
\end{aligned}
\end{equation*}
Now fix $0<\r<R$. We see from \eqref{lim g} that there exists $N$ satisfying
\begin{equation*}
D(s)=\int_{0}^{s}g^{2}(R,t) dt \leq \int_{0}^{N}g^{2}(R,t)dt+C\int_{N}^{s}t^{2p} dt \leq  Cs^{2p+1}.
\end{equation*}
for all $s>N$. So for large enough $s$ we have that $s\leq D^{-1}(Cs^{sp+1})$, which then implies
\begin{equation*}
\begin{aligned}
g\left(R,D^{-1}\left(C\int_{0}^{2s}F(\r,t)dt\right)\right)&\geq g\left(R,D^{-1}\left(C\left(\int_{0}^{s}\sqrt{f(\r,t)}dt\right)^{2}\right)\right)\\
&= g\left(R,D^{-1}\left(\left(C\left(\int_{0}^{s}\sqrt{f(\r,t)}dt\right)^{\frac{2}{2p+1}}\right)^{2p+1}\right)\right)\\
&\geq g\left(R,C\left(\int_{0}^{s}\sqrt{f(\r,t)}dt\right)^{\frac{2}{2p+1}}\right)\\
&\geq C\left(\int_{0}^{s}\sqrt{f(\r,t)}dt\right)^{\frac{2p}{2p+1}},
\end{aligned}
\end{equation*}
where the first inequality follows from Lemma 4.1 in Singh 2015. By $(i)$ and $(ii)$ of theorem \ref{newthm} we thus see that if our non-linearity $g(r,t)$ satisfies \eqref{lim g}, then
\begin{equation*}
\lim\limits_{r\to R}v(r)=\infty\quad\mbox{if and only if}\quad \int_{1}^{\infty}\frac{ds}{\left(\int_{0}^{s}\sqrt{f(\r,t)}dt\right)^{\frac{2p}{2p+1}}}<\infty 
\end{equation*}
for some $\r\leq R$. 
}
\end{remark}

\smallskip

Next, if $g(|x|,v(x)) = |x|^{a} v^p$ and $f(|x|, |\nabla u (x)|) = |x|^{b} h(|\nabla u|)$, where $h\in C^{1}[0,\infty)$ is an increasing function such that $h(t) > 0$ for all $t > 0$. Then we have the following result:
\begin{theorem}\label{thm1}
Let us suppose that $(u, v)$ is a positive radial solution of
\begin{equation}\label{osystem}
\left\{
\begin{aligned}
\Delta u&=|x|^{a}v^{p}  &&\quad\mbox{ in } B_R, \\
\Delta v&=|x|^{b}h(|\nabla u|) &&\quad\mbox{ in } B_R,
\end{aligned}
\right.
\end{equation}
\begin{enumerate}
\item[(i)] Both $u$ and $v$ are bounded if and only if
\begin{equation}\label{bounded}
\int_{1}^\infty\frac{ds}{\Big(\displaystyle \int_0^s H(t)dt  \Big)^{p/(2p+1)}} = \infty.
\end{equation}
\item[(ii)] $u$ is bounded and $\lim_{|x|\nearrow R}v(r)=\infty$  if and only if
\begin{equation}\label{int1}
\int_{1}^\infty\frac{s ds}{\Big(\displaystyle \int_0^s H(t)dt  \Big)^{p/(2p+1)}} <\infty.
\end{equation}
\item[(iii)] $\lim_{|x|\nearrow R}u(x)=\lim_{|x|\nearrow R}v(x)=\infty$ if and only if
\begin{equation}\label{int2}
\int_{1}^{\infty}\frac{ds}{\left(\int_{0}^{s}H(t)dt\right)^{\frac{p}{2p+1}}} <\infty\quad\mbox{and} \int_{1}^{\infty}\frac{s\ ds}{\left(\int_{0}^{s}H(t)dt\right)^{\frac{p}{2p+1}}}=\infty.
\end{equation}
\end{enumerate}
\end{theorem}

{\bf Note: } Here, $H(t) = \int_0^t h(k)dk$.

\smallskip

If $h(t) = t^s$, $s\geq 1$, then we get the following result:

\begin{cor}
Consider the system
\begin{equation}\label{system 3}
\left\{
\begin{aligned}
\Delta u &=|x|^av^p &&\quad\mbox{in}\ B_{R}, \\
\Delta v &=|x|^{b}|\nabla u|^{s} &&\quad\mbox{in}\ B_{R}.
\end{aligned}
\right.
\end{equation}
We then have:
\begin{enumerate}[label=(\roman*)]
\item All positive radial solutions to \eqref{system 3} are bounded if and only if
\begin{equation*}
ps-1\leq 0.
\end{equation*}
\item There exists positive radial solutions to \eqref{system 3} satisfying \eqref{u bounded} and $\lim\limits_{|x|\to R}v(x)=\infty$ if and only if
\begin{equation*}
s>2\left(1+\frac{1}{p}\right).
\end{equation*}
\item There exists positive radial solutions to \eqref{system 3} satisfying \eqref{u,v unbound} if and only if
\begin{equation*}
\frac{1}{p}<s\leq2\left(1+\frac{1}{p}\right).
\end{equation*}
\end{enumerate}
\end{cor}

Now, we study the system \eqref{system 3} in the whole space $\bR^N$, that is,

\begin{equation}\label{sysrn}
\left\{
\begin{aligned}
\Delta u&= |x|^{a}v^p&&\quad\mbox{ in } \bR^N,\\
\Delta v&= |x|^{b}|\nabla u|^s &&\quad\mbox{ in } \bR^N,
\end{aligned}
\right.
\end{equation}
where $a$, $b$, $p> 0$ and $s \geq 1$.

\begin{theorem}\label{thmrn}
\begin{enumerate}
\item System \eqref{sysrn} has positive radial solutions if and only if
\begin{equation}\label{int sqrt f inf}
ps-1\leq 0.
\end{equation}

\item Assume $p< 1$ and $ps < 1$. Let $(u,v)$ be a positive radially symmetric solution of \eqref{sysrn}. If 
\begin{equation}\label{div}
\frac{p(s-2)(s+as+b+2)}{1-ps} \leq 2(N+a-1),
\end{equation}
then
\begin{equation*}
\lim_{|x|\rightarrow \infty}\frac{u(x)}{|x|^{\frac{(a+2)(1-ps)+ ps(a+1)+bp}{1-ps}}}= \frac{(AB^s K)^{\frac{p}{ps - 1}}}{DK}
\end{equation*}
and
\begin{equation*}
\lim_{|x|\rightarrow \infty}\frac{v(x)}{|x|^{\frac{(a+1)s + b + 2}{1-ps}}}= (AB^s K)^{\frac{1}{ps - 1}},
\end{equation*}
where, 
$$
A = 2 + \frac{b+s(1+a+2p)}{1-ps},
$$
	
$$
B = N+a + p\Big(2+\frac{b+s(1+a+2p)}{1-ps}\Big),
$$
	
$$
K = N + \frac{b+s(1+a+2p)}{1-ps},
$$
	
$$
D = 2+a + p\Big(2+\frac{b+s(1+a+2p)}{1-ps}\Big).
$$
\end{enumerate}
\end{theorem}

\smallskip

\section{Notes on dynamical systems}
Let $x=(x_1,x_2,x_3)$, $y=(y_1,y_2,y_3)$ be any two points in $\bR^{3}$, then for $x_i\leq y_i$ we write
$$
x\leq y 
$$
where $i=1,2,3$. For $x\leq y$ and $x\neq y$, we write $x<y$.

Also, the open ordered interval is defined as
$$
[[x,y]]=\{z\in \bR^3:x<z<y\}\subset \bR^3.
$$
Consider the initial value problem
\begin{equation}\label{det1}
	\left\{
	\begin{aligned}
		&\xi_{t}=H(\xi) \quad\mbox{ for } t\in \bR,\\
		&\xi(0)=\xi_{0},
	\end{aligned}
	\right.
\end{equation}
where $ H:\bR^{3}\rightarrow \bR$ is a $C^{1}$ function. This implies that there exists a unique solution $\xi$ of \eqref{det1} defined in a maximal time interval for any $\xi_{0}\in \bR^{3}$.  Let $\varphi(\cdot,\xi_{0})$ denotes the flow associated to \eqref{det1}, that is,  $t\longmapsto \varphi(t,\xi_{0})$ is the unique solution of \eqref{det1} defined in a maximal time interval. Let us suppose that the vector field $h$ is cooperative, that is
$$
\frac{\partial H_{i}}{\partial x_{j}}\geq 0 \quad \mbox{ for } 1\leq i,j\leq 3,\;\; i\neq j.
$$
Next, follow the results due to Hirsch \cite{Hirsch1989, Hirsch1990}.
\begin{theorem}\label{thmdet1}{\rm (see \cite[Theorem 1]{Hirsch1990})}
	Any compact limit set of \eqref{det1} contains an equilibrium or is a cycle.
\end{theorem}
\begin{definition}
	A finite sequence of equilibria $\zeta_{1},\zeta_{2},\dots,\zeta_{n}=\zeta_{1}$, $(n\geq 2)$ is known as a circuit such that $W^{u}(\xi_{i})\cap W^{s}(\xi_{i+1})$ is non-empty. Here, $W^{u}$ and $W^{s}$ respresents the stable and unstable manifolds respectively.
\end{definition}

\medskip
{\bf{Note: }}	There is no circuit in case all the equilibria are hyperbolic and also their stable and unstable manifolds are mutually transverse.

\begin{theorem}\label{thmdet2}{\rm (see \cite[Theorem 2]{Hirsch1990}). }
	Let us assume that $L\subset \bR^{3}$ is a compact set such that:
	\begin{enumerate}
		\item [ (i) ] There is no circuit and all the equilibria in $L$ are hyperbolic.
		\item [ (ii)] The number of cycles in $L$ which have period less than or equal to $T$ is finite, where $T>0$.
	\end{enumerate}
	Then:
	\begin{enumerate}
		\item [ (a) ] Every limit set in $L$ is an equilibrium or cycle.
		\item [ (b) ] $L$ has finite number of cycles.
	\end{enumerate}
	
\end{theorem}

\begin{theorem}\label{thmdet3}{\rm (see \cite[Theorem 7]{Hirsch1989})}
Assume that $\xi_{1}$, $\xi_{2}\in \bR^{3}$ such that $\xi_1< \xi_2$. Further, if
	$$
	{\rm div}{H}<0 \quad\mbox{ in } [[\xi_1,\xi_2]],
	$$
	then there are no cycles of \eqref{det1} in $[[\xi_1,\xi_2]].$
\end{theorem}

\smallskip

\section{Proof of the Theorem \ref{newthm}}

\begin{itemize}
\item [(i) ]System \eqref{system} can be rewritten as
\begin{equation*}
\left\{
\begin{aligned}
(u'(r)r^{N-1})'&= r^{N-1}g(r,v),\\
(v'(r)r^{N-1})'&= r^{N-1}f(r,|\nabla u|),\\
v'(0)&=u'(0)=0.
\end{aligned}
\right.
\end{equation*}
An integration over $[0,r]$ then gives us
\begin{equation*}
\left\{
\begin{aligned}
u'(r)= r^{1-N}\int_{0}^{r}t^{N-1}g(t,v(t))dt,\\
v'(r)=r^{1-N}\int_{0}^{r}t^{N-1}f(t,|\nabla u(t)|)dt.
\end{aligned}
\right.
\end{equation*}
From which it follows that $u$ and $v$ are increasing. We can thus take the first integral equation above and estimate $u'$ as 
\begin{equation*}
u'(r)\leq r^{1-N}g(r,v(r))\int_{0}^{r}t^{N-1}dt=\frac{rg(r,v(r))}{N}
\end{equation*}
so from \eqref{system} we have that
\begin{equation*}
\begin{aligned}
g(r,v(r))&\leq u''(r)+\frac{N-1}{r}\cdot \frac{rg(r,v(r))}{N}\\
\imp \frac{g(r,v(r))}{N} &\leq u''(r).
\end{aligned}
\end{equation*}
Similarly we have
\begin{equation*}
v'(r)\leq \frac{rf(r,u'(r))}{N},
\end{equation*}
from which it follows
\begin{equation*}
\frac{f(r,u'(r))}{N} \leq v''(r).
\end{equation*}
We thus have the following two estimates, which will be important as we proceed:
\begin{equation}\label{u'' bound}
\frac{g(r,v(r))}{N} \leq u'' \leq g(r,v(r)),
\end{equation}
and
\begin{equation}\label{v'' bound}
\frac{f(r,u'(r))}{N} \leq v'' \leq f(r,u'(r)). 
\end{equation}
Now, where $w=u'$, we multiply the right inequality in \eqref{v'' bound} by $v'$ and integrate over $(0,r)$ to find
\begin{equation}\label{v' bound}
\begin{aligned}
\frac{v'(r)^{2}}{2} &\leq \int_{0}^{r} f(t,w(t))v'(t) dt \nonumber\\
&\leq f(r,w(r))\int_{0}^{r}v'(t)dt\nonumber\\
&\leq f(R,w(r))v(r),
\end{aligned}
\end{equation}
from which it follows
\begin{equation}\label{root f bound}
\frac{v'(r)}{\sqrt{v(r)}}\leq \sqrt{2}\sqrt{f(R,w(r))}.
\end{equation}
Now, fix $0<r_0<R$. From \eqref{u'' bound}, we see that for all $r>r_0$ we have
\begin{equation*}
g(r_{0},v(r))\leq Nw'(r),
\end{equation*}
and multiplying this inequality by \eqref{root f bound} we obtain
\begin{equation*}
\frac{v'(r)g(r_{0},v(r))}{\sqrt{v(r)}}\leq N\sqrt{2}\sqrt{f(R,w(r))}w'(r).
\end{equation*}
Integrating both sides of this equation over $(r_0,r)$, and changing variables gives
\begin{equation*}
\begin{aligned}
\int_{r_{0}}^{r}\frac{v'(t)g(r_{0},v(t))}{\sqrt{v(t)}}dt=\int_{v(r_{0})}^{v(r)}\frac{g(r_{0},t)}{\sqrt{t}}dt&\leq N\sqrt{2}\int_{r_0}^{r}\sqrt{f(R,w(t))}w'(t)dt\\
&\leq N\sqrt{2}\int_{w(r_0)>0}^{w(r)}\sqrt{f(R,t)}dt\\
&\leq N\sqrt{2}\int_{w(0)=0}^{w(r)}\sqrt{f(R,t)}dt\quad\mbox{for all}\ r>r_0.
\end{aligned}
\end{equation*}
From \eqref{u'' bound} we have that $w'(r)\leq g(R,v(r))$ for all $r>r_{0}$.
We now recall that
\begin{equation*}
	A_{r_{0}}(s)=\int_{0}^{s}\frac{g(r_{0},t)}{\sqrt{t}}dt,
	\end{equation*}
	is increasing, which gives us
	\begin{equation}\label{w' int bound}
		\begin{aligned}
			w'(r)\leq g(R,v(r))&=g\left(R, A_{r_{0}}^{-1}\left(\int_{0}^{v(r)}\frac{g(r_{0},t)}{\sqrt{t}}dt\right)\right)\nonumber\\
			&\leq g\left(R, A_{r_{0}}^{-1}\left(C\int_{0}^{w(r)}\sqrt{f(R,t)}dt\right)\right),
		\end{aligned}
	\end{equation}
	which then implies
	\begin{equation*}
		\frac{w'(r)}{g\left(R, A_{r_{0}}^{-1}\left(C\int_{0}^{w(r)}\sqrt{f(R,t)}dt\right)\right)}\leq C.
	\end{equation*}
	An integration over $(r_{0},r)$ then gives
	\begin{equation*}
		\int_{r_{0}}^{r}\frac{w'(t)}{g\left(R, A_{r_{0}}^{-1}\left(C\int_{0}^{w(t)}\sqrt{f(R,s)}ds\right)\right)}dt\leq C(r-r_{0})\leq Cr,
	\end{equation*}
	or, after a change of variables,
	\begin{equation*}
		\int_{w(r_{0})}^{w(r)}\frac{ds}{g\left(R, A_{r_{0}}^{-1}\left(C\int_{0}^{s}\sqrt{f(R,t)}dt\right)\right)}\leq Cr.
	\end{equation*}
	Letting $r\to R$, we see
	\begin{equation*}
		\int_{w(r_{0})}^{\infty}\frac{ds}{g\left(R, A_{r_{0}}^{-1}\left(C\int_{0}^{s}\sqrt{f(R,t)}dt\right)\right)}\leq C R<\infty,
	\end{equation*}
	and the result follows.

\smallskip

\item[(ii)]
	We see that \eqref{system} can be rewritten as
	\begin{equation}\label{new system}
		\left\{
		\begin{aligned}
			u(r)&=u(0)+\int_{0}^{r}t^{1-N}\left(\int_{0}^{t}g(s,v(s))ds\right)dt &&\quad\mbox r>0,\\
			v(r)&=v(0)+\int_{0}^{r}t^{1-N}\left(\int_{0}^{t}f(s,|u'(s)|)ds\right)dt &&\quad\mbox r>0,\\
			u(0)&>0 &&\quad\mbox v(0)>0.
		\end{aligned}
		\right.
	\end{equation}
	Using a contraction mapping argument, we can show that system \eqref{new system} has a solution $(u,v)$ defined on some maximum interval $[0,R_{0})$. Using Lemma 4.1 in~\cite{singh2015}, we see that for each $\r$
	\begin{equation}\label{F integral}
		g\left(R_{0}, C\left(\int_{0}^{s}\sqrt{f(\r,t)}dt\right)^{2}\right)\leq g\left(R_{0},2C\int_{0}^{s}F(\r,t)dt\right),
	\end{equation}
	where 
	\begin{equation*}
		F(r,t)\coloneqq \int_{0}^{t}f(r,\s)d\s.
	\end{equation*}
	Now, fix $\r\in(0,R_{0})$, and, recalling \eqref{u'' bound} and \eqref{v'' bound}, we have for all $\r\leq r< R_{0}$
	\begin{equation}\label{f,w' bounds}
		\begin{aligned}
			f(\r,w(r))&\leq Nv''(r),\\
			w'(r)&\leq g(R_{0},v(r)).
		\end{aligned}
	\end{equation}
	Multiplying the two inequalities in \eqref{f,w' bounds} and the integrating over $[\r,r]$ we find
	\begin{equation*}
		F(\r,w(r))-F(\r,w(\r))\leq Ng(R_{0},v(r))v'(r)
	\end{equation*}
	or, in other words,
	\begin{equation*}
		F(\r,w(r))\leq Cg(R_{0},v(r))v'(r)\quad\mbox{for all}\ \r\leq r<R_{0}.
	\end{equation*}
	Using \eqref{f,w' bounds} again, we see that the above becomes
	\begin{equation*}
		w'(r)F(\r,w(r))\leq Cg^{2}(R_{0},v(r))v'(r)\quad\mbox{for all}\ \r\leq r<R_{0}.
	\end{equation*}
	Define
	\begin{equation*}
		G(r)\coloneqq\int_{w(\r)}^{r}F(\r,t)dt\quad\mbox{for all}\ \r\leq r<R_{0}.
	\end{equation*}
	We thus have
	\begin{equation*}
		\begin{aligned}
			G(w(r))=\int_{w(\r)}^{w(r)}F(\r,t)dt&\leq C\int_{\r}^{r}g^{2}(R_{0}, v(t))v'(t)dt\\
			&=C\int_{v(\r)}^{v(r)}g^{2}(R_{0},t)dt\\
			&\leq CD(v(r)),
		\end{aligned}
	\end{equation*}
	which then yields
	\begin{equation*}
		\int_{0}^{w(r)}F(\r,t)dt=\int_{0}^{w(\r)}F(\r,t)dt+\int_{w(\r)}^{w(r)}F(\r,t)dt=C+\int_{w(\r)}^{w(r)}F(\r,t)dt\leq CD(v(r)),
	\end{equation*}
	and since $D^{-1}$ is an increasing function, we thus have (recalling $w'(r)\leq g(R_{0}, v(r)) $)
	\begin{equation*}
		C\leq \frac{w'(r)}{g(R_{0},D^{-1}(C\int_{0}^{w(r)}F(\r,t)dt))}.
	\end{equation*}
	Integrating the above over $[\r,r]$ gives
	\begin{equation*}
		C(r-\r)\leq \int_{\r}^{r}\frac{w'(t)\ dt}{g(R_{0},D^{-1}(C\int_{0}^{w(t)}F(\r,s)ds))}=\int_{w(\r)}^{w(r)}\frac{ds}{g(R_{0},D^{-1}(C\int_{0}^{s}F(\r,t)dt))}
	\end{equation*}
	Now, letting $r\to R_{0}$ we see that
	\begin{equation*}
		R_{0}\leq C\int_{w(\r)}^{\infty}\frac{ds}{g(R_{0},D^{-1}(C\int_{0}^{s}F(\r,t)dt))}\leq C\int_{1}^{\infty}\frac{ds}{g(R_{0},D^{-1}(C\int_{0}^{s}F(\r,t)dt))}  <\infty.
	\end{equation*}
	We have obtained a positive radial solution $(u,v)$ of \eqref{system} in $B_{R_{0}}$ satisfies $\lim_{r\to R_{0}}v(r)=\infty$. Now, if $R>0$ is any arbitrary radius, we set
	\begin{equation*}
		\begin{aligned}
			\tilde{f}\left(r,t\right)&=\l^{2}f\left(\l r,\frac{t}{\l}\right),\\
			\tilde{g}(r,t)&=\l^2g(\l r,t).
		\end{aligned}
	\end{equation*}
	By the above, we know there exists $(\tilde{u},\tilde{v})$ satisfying
	\begin{equation*}
		\left\{
		\begin{aligned}
			\Delta \tilde{u}&= \tilde{g}(r,\tilde{v(r)}) &&\quad\mbox{in}\ B_{R_{0}},\\
			\Delta \tilde{v}&= \tilde{f}(r,|\nabla \tilde{u(r)}|) &&\quad\mbox{in}\ B_{R_{0}},
		\end{aligned}
		\right.
	\end{equation*}
	where $B_{R_{0}}$ is a maximum ball of existence. Let
	\begin{equation*}
		\left\{
		\begin{aligned}
			u(r)&= \tilde{u}\left(\frac{r}{\l}\right) &&\quad\mbox{in}\ B_{R},\\
			v(r)&= \tilde{v}\left(\frac{r}{\l}\right) &&\quad\mbox{in}\ B_{R}.
		\end{aligned}
		\right.
	\end{equation*}
	Taking $\l=\frac{R}{R_{0}}$ we see that $(u,v)$ is a solution to \eqref{system} in $B_{R}$.

\smallskip

\item[(iii)]
	First, we note that one can proceed as in the proof of Lemma 2 to obtain the local existence of a solution, and then use the scaling argument from Lemma 2 to obtain a solution in a ball of radius $R$. Arguing in a similar way to Lemmas 1 and 2, we find that there exists $r_{0},\r\in(0,R)$ such that
	\begin{equation}\label{int comparison}
		\int_{w(r)}^{\infty}\frac{ds}{g\left(R, A_{r_{0}}^{-1}\left(C_{1}\int_{0}^{s}\sqrt{f(R,t)}dt\right)\right)}\leq C_{2}(R-r)
	\end{equation}
	and
	\begin{equation}\label{int lower bound}
		\int_{w(r)}^{\infty}\frac{ds}{g(R,D^{-1}(C_{3}\int_{0}^{s}F(\r,t)dt))}\geq C_{4}(R-r)
	\end{equation}
	for all $\r<r<R$, where $C_{i}>0$ for $i=1,...,4$.  Now, let $\Psi,\Theta:(0,\infty)\rightarrow (0,\infty)$ be defined as
	\begin{equation*}
		\begin{aligned}
			\Psi(t)&=\int_{t}^{\infty}\frac{ds}{g(R,D^{-1}(C_{3}\int_{0}^{s}F(\r,t)dt))},\\ \Theta(t)&= \int_{t}^{\infty}\frac{ds}{g\left(R, A_{r_{0}}^{-1}\left(C_{1}\int_{0}^{s}\sqrt{f(R,t)}dt\right)\right)}.
		\end{aligned}
	\end{equation*}
	We note that $\Phi,\Theta$ are decreasing and we see $\lim_{t\to \infty}\Theta(t)=\lim_{t\to \infty}\Phi(t)=0$. From \eqref{int comparison}and \eqref{int lower bound} we find
	\begin{equation*}
		\Theta(w(r))\leq C_{2}(R-r)\quad\mbox{and}\quad\ \Psi(w(r))\geq C_{4}(R-r)\quad\mbox{for all}\ \r\leq r<R. 
	\end{equation*}
	Since $\Phi,\Theta$ are decreasing, this then implies
	\begin{equation*}
		\begin{aligned}
			w(r)&\geq \Theta^{-1}(C_{1}(R-r)) &&\quad\mbox{for all}\ \r\leq r<R,\\
			w(r)&\leq \Psi^{-1}(C_{2}(R-r)) &&\quad\mbox{for all}\ \r\leq r<R.
		\end{aligned}
	\end{equation*}
	Now, recalling that
	\begin{equation*}
		u(r)=u(\r)+\int_{\r}^{r}w(t)dt\quad\mbox{for all}\ \r\leq r<R,
	\end{equation*}
	we see that $\lim_{r\to R}u(r)=\infty$ if and only if
	\begin{equation*}
		\int_{\r}^{R}w(t)dt=\infty,
	\end{equation*}
	which implies
	\begin{equation*}
		\int_{\r}^{R}\Psi^{-1}(C(R-r))dt=\infty,
	\end{equation*}
	for some $C>0$. Hence $\lim_{r\to R}u(r)=\infty$ only if
	\begin{equation*}
		\int_{0}^{C(R-\r)}\Psi^{-1}(u)du=\infty,
	\end{equation*}
	which is true if and only if
	\begin{equation*}
		\int_{0}^{1}\Psi^{-1}(u)du=\infty.
	\end{equation*}
	The change of variables $s=\Psi^{-1}(u)$ then gives us that $\lim_{r\to R}u(r)=\infty$ only if
	\begin{equation*}
		\int_{1}^{\infty}\frac{s\ ds}{g(R,D^{-1}(C_{3}\int_{0}^{s}F(\r,t)dt))}=\infty.
	\end{equation*}
	The proof of the second statement follows a similar argument to above.
\end{itemize}

\smallskip

\section{Proof of the Theorem \ref{thm1}}
System \eqref{osystem} can be rewritten as
\begin{equation*}
\left\{
\begin{aligned}
&(u'(r)r^{N-1})'= r^{N+a-1}v^{p}(r),\\
&(v'(r)r^{N-1})'= r^{N+b-1}h(|\nabla u|),\\
&u'(0)=v'(0)=0.
\end{aligned}
\right.
\end{equation*}
An integration over $(0,r)$, $0<r<R$, then gives us
\begin{equation*}
\left\{
\begin{aligned}
u'(r)&= r^{1-N}\int_{0}^{t}t^{N+a-1}v^{p}(t)dt,\\
v'(r)&=r^{1-N}\int_{0}^{t}t^{N+b-1}h(|\nabla u|)dt,
\end{aligned}
\right.
\end{equation*}
from which it follows that $u$ and $v$ are increasing in $(0,R)$. We can thus take the first integral equation above and estimate $u'$ as 
\begin{equation*}
u'(r)\leq r^{1-N}v^{p}(r)\int_{0}^{r}t^{N+a-1}dt=\frac{r^{a+1}v^{p}(r)}{N+a}
\end{equation*}
so from \eqref{osystem} we have that
\begin{equation*}
r^{a}v^{p}(r)\leq u''(r)+\frac{N-1}{r}\cdot \frac{r^{a+1}v^{p}(r)}{N+a},
\end{equation*}
which further implies,
\begin{equation*}
	\frac{1+a}{N+a}r^{a}v^{p}(r) \leq u''(r).
\end{equation*}
Similarly we have
\begin{equation*}
v'(r)\leq \frac{r^{b+1}h(u'(r))}{N+b},
\end{equation*}
from which it follows
\begin{equation*}
	\frac{1+b}{N+b}r^{b}h(u'(r)) \leq v''(r).
\end{equation*}
We thus have the following two estimates for all $0<r<R$, 
\begin{equation}\label{u'' bound}
\frac{1+a}{N+a}r^{a}v^{p}(r)\leq u'' \leq r^{a}v^{p}(r), 
\end{equation}

\begin{equation}\label{v'' bound}
\frac{1+b}{N+b}r^{b}h(u'(r))\leq v'' \leq r^{b}h(u'(r)). 
\end{equation}

The remaining of the proof will follow on the similar lines as the proof of Theorem 1 in \cite{S2015}

\section{Proof of the Theorem \ref{thmrn}}
We obtained that $u'$, $v'$, $u$, $v$ are increasing in the proof of the Theorem \ref{thm1} and 
\begin{equation*}
\left\{
\begin{aligned}
&u'(r)=r^{1-N}\int_{0}^{r}{t^{N-1+a}v^{p}(t)}dt\quad\mbox{ for all } r>0,\\
&v'(r)=r^{1-N}\int_{0}^{r}{t^{N-1+b}(u')^{s}(t)}dt \quad\mbox{ for all } r>0,
\end{aligned}
\right.
\end{equation*}
which gives us
\begin{equation}\label{rn1}
	\frac{r^{a+1}v^{p}(0)}{N+a}\leq u'(r)\leq \frac{r^{a+1}v^{p}(r)}{N+a} \quad\mbox{ for all } r>0
\end{equation}
and
\begin{equation}\label{rn2}
	\frac{v^{ps}(0)r^{(a+1)s+b+1}}{(N+b)(N+a)^s}\leq v'(r)\leq \frac{r^{b+1}u'^{s}(r)}{N+b} \quad\mbox{ for all } r>0.
\end{equation}
Using \eqref{rn1} and \eqref{rn2} we obtain that $u'(r)$, $v'(r)$, $u(r)$, $v(r)$ tend to infinty as $r\rightarrow \infty$. Next, we do the following change of variables ( see \cite{HV1996}, \cite{BVH2010,G2012} )
$$
X(t)= \frac{ru'(r)}{u(r)},\;\;Y(t)= \frac{rv'(r)}{v(r)},\;\;Z(t)=\frac{r^{a+1}v^{p}(r)}{u'(r)} \mbox{ and }W(t)=\frac{r^{b+1}u'^{s}(r)}{v'(r)}, 
$$
where $t= \ln(r)$ for  $r\in (0,\infty)$. Direct computation shows that $(X(t),Y(t),Z(t),W(t))$ satisfies
\begin{equation}\label{rn3}
\left\{
\begin{aligned}
		&X_{t}= X(Z-(N-2)-X) \quad\mbox{ for all } t\in \bR, \\
		&Y_{t}= Y(W-(N-2)-Y) \quad\mbox{ for all } t\in \bR, \\
		&Z_{t}= Z(N+a+pY-Z) \quad\mbox{ for all } t\in \bR, \\
		&W_{t}= W(sZ+N-sN+s+b-W) \quad\mbox{ for all } t\in \bR.
\end{aligned}
\right.
\end{equation}
Also, by L'Hopital's rule we deduce that $\lim_{t\rightarrow \infty}X(t)=2-N+\lim_{t\rightarrow \infty}Z(t)$. Hence, it is enough to study the last three equations of \eqref{rn3}, that is, 
\begin{equation}\label{rn4}
\left\{
\begin{aligned}
&Y_{t}= Y(W-(N-2)-Y) \quad\mbox{ for all } t\in \bR, \\
&Z_{t}= Z(N+a+pY-Z) \quad\mbox{ for all } t\in \bR, \\
&W_{t}= W(sZ+N-sN+s+b-W) \quad\mbox{ for all } t\in \bR.
\end{aligned}
\right.
\end{equation}
Our system can be rewritten as
\begin{equation}\label{rn5}
\xi_{t}= H(\xi)
\end{equation}
where
$$
\xi=\left(\begin{array}{c}Y(t)\\Z(t)\\W(t)\end{array}\right) \quad\mbox{ and } \quad H(\xi)= \left(\begin{array}{c} Y(W-(N-2)-Y)\\Z(N+a+pY-Z)\\ W(sZ+N-sN+s+b-W) \end{array}\right).
$$

One notices that the system \eqref{rn5} is cooperative. Therefore, the following comparison principle holds:
\begin{lemma}\label{lrn1}
	Let us assume that $\xi(t)= \left(\begin{array}{c}Y(t)\\Z(t)\\W(t)\end{array}\right)$ and $\tilde{\xi}(t)= \left(\begin{array}{c}\tilde{Y}(t)\\ \tilde{Z}(t)\\ \tilde{W}(t)\end{array}\right)$ are the two nonnegative solutions of \eqref{rn5} such that
	$$
	Y(t_0)\geq \tilde{Y}(t_0), \;\;\;Z(t_0)\geq \tilde{Z}(t_0),\;\;\; W(t_0)\geq \tilde{W}(t_0)
	$$
	for some $t_0\in \bR$. Then, we have
	$$
	Y(t)\geq \tilde{Y}(t), \;\;\;Z(t)\geq \tilde{Z}(t),\;\;\; W(t)\geq \tilde{W}(t) \quad\mbox{ for all }t\geq t_{0}.
	$$
\end{lemma}

\bigskip

\bigskip

Using \eqref{rn1} and \eqref{rn2} we deduce that $Z\geq N+a$ and $W\geq N+b$. Hence, we only have two equilibria of \eqref{rn4} which satisfy $Z\geq N+a$ and $W\geq N+b$, that is,
$$
\xi_{1}= \left(\begin{array}{c}0\\N+a\\N+s(a+1)+b\end{array}\right) \quad\mbox{  and }
\xi_{2}=\left(\begin{array}{c}2 + \frac{b+s(1+a+2p)}{1-ps},\\N+a + p\Big(2+\frac{b+s(1+a+2p)}{1-ps}\Big)\\N + \frac{b+s(1+a+2p)}{1-ps}\end{array}\right).
$$
\begin{lemma}\label{lrn2}
$\xi_{2}$ is asymptotically stable.
\end{lemma}
\begin{proof}
At $\xi_2$, we have the following linearized matrix: 
$$
M=\left[\begin{array}{ccc}-Y_2&0&Y_2\\pZ_2&-Z_2&0\\0&sW_2&-W_2\end{array}\right],
$$
and the associated characteristic polynomial of $M$ is
$$
P(\lambda)=\det(\lambda I-M)=\lambda^{3}+\alpha\lambda^{2}+\beta\lambda+(1-ps)\gamma,
$$
where
\begin{equation*}
\begin{aligned}
\alpha&=Y_2+Z_2+W_2\\
\beta&=Y_2Z_2+Z_2W_2+Y_2W_2\\
\gamma&=Y_2Z_2W_2.
\end{aligned}
\end{equation*}
As $ps<1$ and  $\alpha$, $\beta$, $\gamma>0$, we have that $P(\lambda)>0$ for all $\lambda \geq 0$. In case $P$ has three real roots, then all of them are negative, which makes $\xi_2$ asymptotically stable. Next, we need to consider the case where $P$ has exactly one real root. So, let us assume that $\lambda_1 \in \bR$ and $\lambda_2,\lambda_3 \in \mathbb{C}\setminus \bR$ be the roots of characteristic polynomial  $P$. We need to show that ${\rm Re}(\lambda_2)={\rm Re}(\lambda_3)<0$, that is, $P(-\alpha)=-\beta\alpha+(1-ps)\gamma<0$, which is same as claiming, $\beta\alpha>(1-ps)\gamma$. By the use of AM-GM inequality we get that
$$
\alpha \geq 3\sqrt[3]{Y_2Z_2W_2}\quad  \mbox{ and } \quad \beta \geq 3\sqrt[3]{(Y_2Z_2W_2)^{2}}
$$
which further gives us the desired result, that is, $\alpha\beta> (1-ps)\gamma$. Hence, $\xi_2$ is asymptotically stable.

\end{proof}

\begin{lemma}\label{lrn3}
	For all $t\in \bR$, we have
	\begin{equation}\label{rn6}
		0\leq Y(t)\leq 2 + \frac{b+s(1+a+2p)}{1-ps},
	\end{equation}
	\begin{equation}\label{rn7}
		N+a\leq Z(t)\leq N+a + p\Big(2+\frac{b+s(1+a+2p)}{1-ps}\Big),
	\end{equation}
	\begin{equation}\label{rn8}
		N+s(a+1)+b\leq W(t)\leq N + \frac{b+s(1+a+2p)}{1-ps}.
	\end{equation}
\end{lemma}
\begin{proof}
As $v'(0)=0$ and $v(0)=0$, we deduce that $\lim_{t\rightarrow -\infty}Y(t)= \lim_{r\rightarrow 0}\frac{rv'(r)}{v(r)}=0$. 
	
Next, we show that there exists $t_{j}\rightarrow -\infty $ such that
\begin{equation}\label{rn9}
\left\{
\begin{aligned}
&Y(t_{j})\leq Y_2,\\
&Z(t_j)\leq Z_2,\\
&W(t_j)\leq W_2.
\end{aligned}
\right.
\end{equation}
Since $\lim_{t\rightarrow -\infty}Y(t)=0$ and $\lim_{t\rightarrow -\infty}Z(t)= N+a$, we only need to prove the last part of \eqref{rn9}. So, let us assume by contradiction that this is not true. Thus $W> W_2$ in $(-\infty,t_0)$ for some $t_0\in \bR$. Then, we have
$$
W_t= W(sZ+N-sN+s+b-W)< 0 \quad\mbox{ in }(-\infty,t_0).
$$
for small enough $t_0$. Therefore, $W$ is decreasing in the neighbourhood of $-\infty$, that is, there exists $\ell= \lim_{t\rightarrow -\infty}W(t)$. Again, by the use L'Hopital's rule we obtain
\begin{equation*}
\begin{aligned}
\ell=\lim_{t\rightarrow -\infty}W(t)&=\lim_{r\rightarrow 0}\frac{r^{b+1}u'^{s}(r)}{v'(r)}\\
&=\frac{s(N+a)-s(N-1)+b+1}{1-\frac{N-1}{\ell}},
\end{aligned}
\end{equation*}
which yields $\ell=N+s(a+1)+b<W_2$ and this contradicts the assumption that $W>W_{2}$ in a neighbourhood of $-\infty$. Hence, with this we have proven the last part of \eqref{rn9}. Next, we apply the Comparison Lemma \ref{lrn1} on all the intervals $[t_j,\infty)$ for $j\geq 1$ in order to obtain the upper bound inequalities in Lemma \ref{lrn3}. Similarly, the lower bound inequalities can be obtained.
\end{proof}

\medskip

Now, assume $L=\overline{[[\xi_{1},\xi_{2}]]}\subset \bR^{3}$. We have that $\omega(\zeta)\subseteq L $ by Lemma \ref{lrn3}. As $\zeta_2$ is asymptotically stable, we get that $L$ has no circuits. Also, by \eqref{div}, one could obtain that
$$
{\rm div}\, H(\xi)=-W+(s-2)Z+(-2+p)Y+2+a+N(1-s)+s+b<0\quad\mbox{ in } L.
$$
By the use of Theorems \ref{thmdet2} and \ref{thmdet3} we get that $\omega(\xi)$ reduces to one of the equilibria $\xi_1$ or $\xi_2$. In case $\xi(t)\rightarrow \xi_1$ as $t\rightarrow \infty$, then we have that $Y(t)\rightarrow 0$ as $t\rightarrow \infty$. Also, we obtain that $Y_t> 0$ in a neighbourhood of infinity by using the second equation of \eqref{rn3} which is impossible given that $Y(t)> 0$ in $\bR$. Therefore, $\xi(t)\rightarrow \xi_2$ as $t\rightarrow \infty$, that is
\begin{equation*}
\begin{aligned}
\lim_{t\rightarrow \infty}X(t)&= 2+a + p\Big(2+\frac{b+s(1+a+2p)}{1-ps}\Big),\\
\lim_{t\rightarrow \infty}Y(t)&= 2 + \frac{b+s(1+a+2p)}{1-ps},\\
\lim_{t\rightarrow \infty}Z(t)&= N+a + p\Big(2+\frac{b+s(1+a+2p)}{1-ps}\Big),\\
\lim_{t\rightarrow \infty}W(t)&= N + \frac{b+s(1+a+2p)}{1-ps}.
\end{aligned}
\end{equation*}
Now , using $(X(t),Y(t),Z(t),W(t))$ , we get
\begin{equation*}
\lim_{|x|\rightarrow \infty}\frac{u(x)}{|x|^{\frac{(a+2)(1-ps)+ ps(a+1)+bp}{1-ps}}}= \frac{(AB^s K)^{\frac{p}{ps - 1}}}{DK}
\end{equation*}
and
\begin{equation*}
\lim_{|x|\rightarrow \infty}\frac{v(x)}{|x|^{\frac{(a+1)s + b + 2}{1-ps}}}= (AB^s K)^{\frac{1}{ps - 1}},
\end{equation*}
where, 
$$
A = \lim_{t\rightarrow \infty}Y(t),
$$

$$
B = \lim_{t\rightarrow \infty}Z(t),
$$

$$
K = \lim_{t\rightarrow \infty}W(t),
$$

$$
D = \lim_{t\rightarrow \infty}X(t).
$$

\end{document}